\documentclass[12pt,reqno]{article}
\usepackage{amsmath,amssymb,amsfonts}
\usepackage[usenames]{color}

\usepackage[colorlinks=true,
linkcolor=webgreen,
filecolor=webbrown,
citecolor=webgreen]{hyperref}

\definecolor{webgreen}{rgb}{0,.5,0}
\definecolor{webbrown}{rgb}{.6,0,0}

\usepackage{color}
\usepackage{fullpage}
\usepackage{float}

\usepackage{graphics,amsmath,amssymb}
\usepackage{amsthm}
\usepackage{amsfonts}
\usepackage{latexsym}
\usepackage{epsf}

\theoremstyle{plain}
\newtheorem{theorem}[equation]{Theorem}
\newtheorem{cor}[equation]{Corollary}

\newtheorem{lemma}[equation]{Lemma}

\theoremstyle{definition}
\newtheorem{definition}[equation]{Definition}

\numberwithin{equation}{section}

\input xy
\xyoption{all}

\newcommand{\seqnum}[1]{\href{http://oeis.org/#1}{\underline{#1}}}

\begin{document}

\begin{center}
\epsfxsize=4in

\end{center}

\begin{center}
\vskip 1cm{\LARGE\bf Action Graphs and Catalan Numbers}
\vskip 1cm
\large
Gerardo Alvarez, Julia E. Bergner, and Ruben Lopez \\
Department of Mathematics \\
University of California, Riverside \\
Riverside, CA 92521 \\
\href{mailto:galva012@ucr.edu}{\tt galva012@ucr.edu} \\
\href{mailto:bergnerj@member.ams.org}{\tt bergnerj@member.ams.org} \\
\href{mailto:rlope015@ucr.edu}{\tt rlope015@ucr.edu}
\end{center}

\vskip .2 in

\begin{abstract}
We introduce an inductively defined sequence of directed graphs and prove that the number of edges added at step $k$ is equal to the $k$th Catalan number.  Furthermore, we establish a bijection between the set of edges adjoined at step $k$ and the set of planar rooted trees with $k$ edges.
\end{abstract}

\section{Introduction}

In a recent paper, the second-named author and Hackney introduced certain inductively defined directed graphs with the goal of understanding the structure of a rooted category action on another category \cite{reedy}.  While these graphs were developed in such a way that they encoded the desired structure, the question remained what kinds of patterns could be found in these inductively defined sequences of graphs.  In this paper, we look at the most basic of these graph sequences.  We begin with the trivial graph with one vertex and no edges, and we inductively add new vertices and edges depending on the number of paths in the previous graph.  We prove that at the $k$th step, the number of vertices and edges added is given by the $k$th Catalan number.

The Catalan numbers give a well-known sequence of natural numbers, arising in many contexts in combinatorics \seqnum{A000108} \cite{oeis}.  A long list of ways to obtain the Catalan numbers is found in Stanley's book \cite{stanley} and subsequent online addendum \cite{catadd}.  Here we use that the 0th Catalan number is $C_0=1$, and , for any $k \geq 1$, the $(k+1)$st Catalan number $C_{k+1}$ is given by the formula
\[ C_{k+1} = \sum_{i=0}^k C_i C_{k-i}. \]

However, it is often of interest to determine a direct comparison with one of the other ways of obtaining the Catalan numbers. To this end, we establish a direct bijection between the set of edges added at step $k$ in the action graph construction and the set of planar rooted trees with $k$ edges, also known to have $C_k$ elements.
 
\section{Action graphs}

We begin by recalling a few basic definitions.

\begin{definition}
A \emph{directed graph} is a pair $G=(V,E)$ where $V$ is a set whose elements are called \emph{vertices} and $E$ is a set of ordered pairs of vertices, called \emph{edges}.  Given an edge $e= (v, w)$, we call $v$ the \emph{source} of $e$, denoted by $v=s(e)$, and call $w$ the \emph{target} of $e$, denoted by $w=t(e)$.
\end{definition}

For the directed graphs that we consider here, we assume that, for every $v \in V$, we have $(v,v) \in E$.  While we could think of these edges as loops at each vertex, we prefer to regard them as ``trivial" edges given by the vertices.  Otherwise, we have no loops or multiple edges in the graphs we consider, so there is no ambiguity in the definition as we have given it.

\begin{definition}
A \emph{(directed) path} in a directed graph is a sequence of edges $e_1, \ldots, e_n$ such that for each $1\leq i<k$, $t(e_i)=s(e_{i+1})$.  For paths consisting of more than one edge, we require all these edges to be nontrivial.  We call $s(e_1)$ the \emph{initial vertex} of the path and $t(e_k)$ the \emph{terminal vertex} of the path.
\end{definition}

The directed graphs we consider here are \emph{labelled} by the natural numbers; in other words, they are equipped with a given function $V \rightarrow \mathbb N$.

\begin{definition}
For each natural number $k$, the \emph{action graph} $A_k$ is the labeled directed graph defined inductively as follows.  The action graph $A_0$ is defined to be the graph with one vertex labeled by 0 and no nontrivial edges.  Inductively, given the $k$th action graph $A_k$, define the $(k+1)$st action graph $A_{k+1}$ by freely adjoining new edges by the following rule.  For any vertex $v$ labeled by $k$, consider all paths in $A_k$ with terminal vertex $v$.  For each such path, adjoin a new edge whose source is the initial vertex $u$ of the path, and whose target is a new vertex which is labeled by $k+1$.
\end{definition}

Thus, the first few action graphs can be depicted as follows:
\[ \xymatrix{A_0: & \bullet_0 && A_1: & \bullet_0 \ar[r] & \bullet_1  \\
A_2: &&& A_3: & \bullet_3 & \bullet_3 \\
& \bullet_2 & \bullet_2 && \bullet_2 \ar[u] & \bullet_2 \ar[u] \\
& \bullet_0 \ar[u] \ar[r] & \bullet_1 \ar[u] & \bullet_3 & \bullet_0 \ar[u] \ar[d] \ar[l] \ar[r] & \bullet_1 \ar[u] \ar[d] \\
&&&& \bullet_3 & \bullet_3} \]

The main result which we wish to prove is the following, which indicates how many new vertices (and edges) are added to $A_k$ to obtain $A_{k+1}$.

\begin{theorem} \label{main}
When building $A_{k+1}$ from $A_k$, the number of vertices added and labeled $k+1$ (and likewise the number of edges added) is given by the $(k+1)$st Catalan number, $C_{k+1}$.
\end{theorem}

We begin with a lemma about paths in action graphs.

\begin{lemma} \label{paths}
The number of paths from 1 to $k$ in $A_k$ is equal to the number of paths from 0 to $k-1$ in $A_{k-1}$.
\end{lemma}

\begin{proof}
In any $A_k$ with $k \geq 1$, there is only one edge connecting the single vertex labeled by 0 and the single vertex labeled by 1.  Since the vertex labeled by 0 is the source and the vertex labeled by 1 is the target, no paths whose source is the vertex labeled by 1 pass through the vertex labeled by 0.  Therefore, edges with source at the vertex labeled by 1 are adjoined exactly in the same way as edges with source 0, but one step later.
\end{proof}

\begin{proof}[Proof of Theorem \ref{main}]
We use induction and Lemma \ref{paths}.  For the base case we know that $A_0$ has one vertex labeled by 0 and $C_0=1$.

For the inductive step, suppose we know $A_k$ has $C_k$ vertices labeled by $k$.  There is always a unique path from 0 to any vertex, so there are $C_k$ paths from the vertex labeled by 0 to a vertex labeled by $k$.  Hence, by considering paths with source the 0 vertex, we create $C_k = C_0 C_k$ new vertices in $A_{k+1}$.  Lemma \ref{paths} tells us that there are $C_{k-1}$ paths from the single vertex labeled by 1 to vertices labeled by $k$.  Thus, we must add $C_{k-1}=C_1 C_{k-1}$ new vertices in $A_{k+1}$.  Applying Lemma \ref{paths} twice, we similarly see that there are $C_{k-2}$ paths from any vertex labeled by 2 (of which there are $C_2$) to any vertex labeled $C_k$, for a total of $C_2C_{k-1}$ new vertices added.  We continue this process, concluding by using the inductive hypothesis that there are $C_k$ vertices labeled by $k$, from each of which there is only one (trivial) path, from which we produce $C_k = C_kC_0$ new vertices.  Taking the sum, we obtain that the total number of paths in $A_k$ ending at vertices labeled by $k$ is
\[ C_0 C_k + C_1 C_{k-1} + \cdots + C_kC_0 = C_{k+1}. \]
Since a new vertex is added for each such path, and a new edge for each such vertex, we have completed the proof.
\end{proof}

The main theorem has the following immediate consequence.

\begin{cor}
The number of vertices in the $k$th action graph $A_k$ is given by the $k$th term in the sequence of partial sums of Catalan numbers \seqnum{A014137} \cite{oeis}.
\end{cor}

\section{A comparison with planar rooted trees}

In this section, we give an explicit one-to-one correspondence between the set of leaves of $A_{k+1}$ and the set of planar rooted trees with $k+1$ edges.  The latter set has $C_{k+1}$ elements \cite[8.4]{koshy}.

We begin with the necessary definitions.  Here, we work with graphs which are no longer assumed to be directed.  A graph is \emph{connected} if there exists a path between any two vertices.

\begin{definition}
A \emph{tree} is a connected graph with no loops or multiple edges.  A \emph{rooted tree} is a tree with a specified vertex called the \emph{root}.
\end{definition}

\begin{definition}
A \emph{leaf} of a rooted tree is a vertex of valence 1 which is not the root.  (In the case of a single vertex with no edges, we take this root vertex also to be a leaf.)  A \emph{branch} of a rooted tree is a path from either the root or from a vertex of valence greater than 2 to a leaf.
\end{definition}

Note in particular that the action graphs are (directed) trees, and that the vertices of $A_{k+1}$ which are not in $A_k$ are precisely the leaves of $A_{k+1}$.

Here we want to consider \emph{planar} rooted trees.  Thus, if we view the bottom vertex as the root, we regard the following two trees as different:
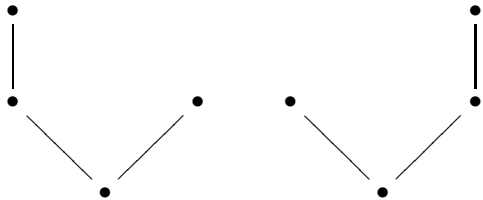
\[ \xymatrix{\bullet &&&&& \bullet \\
\bullet \ar@{-}[u] && \bullet & \bullet && \bullet \ar@{-}[u] \\
& \bullet \ar@{-}[ul] \ar@{-}[ur] &&& \bullet \ar@{-}[ul] \ar@{-}[ur] & } \]
To aid in our comparison with action graphs, we define a means of labeling the vertices of a planar rooted tree.  The root vertex is always given the label 0.  Given a representative of the tree with the root at the bottom, label the vertices by successive natural numbers, moving upward from the root and from left to right.  For example, we have the labelings
\[ \xymatrix{\bullet_2 &&&&& \bullet_3 \\
\bullet_1 \ar@{-}[u] && \bullet_3 & \bullet_1 && \bullet_2 \ar@{-}[u] \\
& \bullet_0 \ar@{-}[ul] \ar@{-}[ur] &&& \bullet_0. \ar@{-}[ul] \ar@{-}[ur] & } \]

Our goal is to prove that we can use these labels to assemble planar rooted trees together to form the action graph in such a way that each planar rooted tree corresponds to exactly one vertex of highest labeling in the action graph.

Observe that the set of all directed trees with $k$ edges can be partially ordered by the length of the unique path from the vertex labeled by 0 to the vertex labeled by $k$.  Beginning with the tree with longest such path length, namely, the tree with a branch of length $k$, for each directed tree, identify the edge from 0 to 1, and any branches that do not end in the vertex $k$, with vertices and edges already present in the graph.  Adjoin new vertices and edges to the graph corresponding to the branch containing the vertex $k$.  Using the partial ordering on the set of trees guarantees that, as we assemble the trees together, each tree that is added contributes only (part of) a branch.  Hence, any possible ambiguity about placement is eliminated.  Repeating for all planar rooted trees with $k$ edges, we claim that the resulting directed graph is precisely the action graph $A_k$.

This assembly can be depicted for the case when $k=2$ as follows.  We two planar rooted trees with two edges are given by
\[ \xymatrix{\bullet_2  &&& \\
\bullet_1 \ar@{-->}[u] & \bullet_1 && \bullet_2 \\
\bullet_0 \ar@{-->}[u] && \bullet_0 \ar@{..>}[ul] \ar@{..>}[ur] & } \]
and can thus be assembled to form the action graph $A_2$, as given by
\[ \xymatrix{\bullet_2 & \bullet_2 \\
\bullet_0 \ar@{..>}[u] \ar@{-->}[r]<.5ex> \ar@{..>}[r]<-.5ex> & \bullet_1. \ar@{-->}[u]} \]

Observe in this example that the overlap of the two trees in $A_2$ corresponds to the action graph $A_1$, and that the two new edges in $A_2$ correspond exactly to the two ways to attach an edge to the unique planar rooted tree with one edge to obtain a planar rooted tree with two edges.

Generalizing this example, we obtain the following explicit correspondence.

\begin{theorem}
The function assigning any planar rooted tree with $k$ edges to the leaf that it contributes to the action graph $A_k$ defines a bijection.
\end{theorem}

\begin{proof}
We prove this theorem inductively.  When $k=0$, both $A_0$ and the only planar rooted tree with no edges consist of a single vertex and no edges.

Thus, suppose that we have proved the result for $k \geq 0$.  Consider the set of planar rooted trees with $k+1$ edges.  By our inductive hypothesis, the trees with $k$ edges assemble to produce the action graph $A_k$.  But these trees can be obtained from those with $k+1$ vertices by removing the leaf labeled by $k+1$.  Indeed, the (possibly multiple) ways of obtaining a rooted tree with $k+1$ edges from one with $k$ edges correspond to the ways we adjoin new leaves to produce $A_{k+1}$ from $A_k$.  In particular, there is precisely one edge ending in a vertex labeled by $k+1$ coming from each planar rooted tree.

Conversely, regard $A_{k+1}$ as a labeled rooted tree.  We claim that the subtrees of $A_{k+1}$ containing exactly one vertex labeled by $i$ for each $0 \leq i \leq k+1$ correspond exactly to the planar rooted trees with $k+1$ edges.  Again, assume this fact is true for $k \geq 0$.  Given every such subtree with $k$ edges, and vertices labeled from 0 to $k$, the choices for adjoining an extra edge containing a vertex labeled by $k+1$ correspond exactly to the vertices on the path from the vertex labeled by 0 to the vertex labeled by $k$.  But, these choices coincide with the ways to obtain a planar rooted tree with $k+1$ edges from a given planar rooted tree with $k$ edges.
\end{proof}

Observe that the action graph $A_k$ can thus be regarded as a universal tree for all planar rooted trees with $k$ edges, in that it is built from these subtrees and that, using the labeling scheme for planar rooted trees, these subtrees can be recovered from $A_k$.

\section{Acknowledgments}

The first- and third-named authors were participants in the RISE program at UC Riverside in Summer 2013, with support from the HSI-STEM and CAMP programs, respectively.  The second-named author was partially supported by NSF grant DMS-1105766 and CAREER award DMS-1352298.  The authors thank the referee for suggestions which improved the exposition of the paper.

\bigskip
\hrule
\bigskip

\noindent 
2010 {\it Mathematics Subject Classification}: Primary 05A19, Secondary 05C05.

\noindent 
\emph{Keywords:} Catalan number, directed graph.

\bigskip
\hrule
\bigskip

\noindent (Concerned with sequences \seqnum{A000108} and \seqnum{A014137}.)

\bigskip
\hrule
\bigskip

\noindent 
Received ; revised version received ; Published in {\it Journal of Integer Sequences} .

\bigskip
\hrule
\bigskip

\noindent 
Return to 
\htmladdnormallink{Journal of Integer Sequences home page}{http://www.cs.waterloo.ca/journals/JIS/}

\bigskip
\hrule
\bigskip

\end{document}